\documentclass{article}
\usepackage{amsmath, amsfonts, amsthm,mathrsfs,comment,color}
\usepackage[all]{xy}

\newcommand{\e}{\varepsilon}
\newcommand{\et}{\text{\rm \'et}}

\newcommand{\id}{\mathrm{id}}
\newcommand{\la}{\langle}
\newcommand{\mc}{\mathcal}
\newcommand{\mf}{\mathfrak}

\newcommand{\msf}{\mathsf}
\newcommand{\ola}{\overleftarrow}

\newcommand{\pr}{\mathrm{pr}}

\newcommand{\ra}{\rangle}
\newcommand{\rig}{\mathrm{rig}}
\newcommand{\Spa}{\mathop{\mathrm{Spa}}\nolimits}
\newcommand{\Spec}{\mathop{\mathrm{Spec}}\nolimits}
\newcommand{\Spf}{\mathop{\mathrm{Spf}}\nolimits}
\newcommand{\til}{\widetilde}

\newcommand{\vani}{\overleftarrow\times}
\newcommand{\wh}{\widehat}
\newcommand{\A}{\mathbb A}
\newcommand{\B}{\mathbb B}
\newcommand{\D}{\mathbb D}
\newcommand{\Q}{\mathbb Q}
\newcommand{\Z}{\mathbb Z}

\newtheorem{thm}{Theorem}[section]

\newtheorem{cor}[thm]{Corollary}
\newtheorem{lem}[thm]{Lemma}
\newtheorem{prop}[thm]{Proposition}
\newtheorem{ques}[thm]{Question}

\theoremstyle{definition}
\newtheorem{const}[thm]{Construction}
\newtheorem{defn}[thm]{Definition}
\newtheorem{exa}[thm]{Example}

\newtheorem{rem}[thm]{Remark}

\makeatletter
  
  \@addtoreset{equation}{section}
\makeatother

\newcommand{\be}{\begin{eqnarray*}}
\newcommand{\ee}{\end{eqnarray*}}
\newcommand{\ben}{\begin{eqnarray}}
\newcommand{\een}{\end{eqnarray}}
\newcommand{\bi}{\begin{itemize}}
\newcommand{\ei}{\end{itemize}}
\newcommand{\bn}{\begin{enumerate}}
\newcommand{\en}{\end{enumerate}}
\newcommand{\bx}{\begin{eqnarray*}\begin{xy}\xymatrix}
\newcommand{\ex}{\end{xy}\end{eqnarray*}}
\newcommand{\bxy}{\begin{xy}\xymatrix}
\newcommand{\exy}{\end{xy}}
\newcommand{\bq}{\begin{ques}}
\newcommand{\eq}{\end{ques}}

\title{\'Etale cohomology of algebraizable rigid analytic varieties via nearby cycles over general bases}
\date{}
\author{Hiroki Kato
\thanks{Universit\'e Paris-Saclay, Laboratoire de Math\'ematiques d'Orsay, F-91405, Orsay, France 
\newline\texttt{Email: hiroki.kato@u-psud.fr}}}
\begin{document}
\maketitle

\begin{abstract}
We prove a finiteness theorem and a comparison theorem in the theory of \'etale cohomology of rigid analytic varieties. 
By a result of Huber, for a quasi-compact separated morphism of rigid analytic varieties with target being of dimension $\le1$, the compactly supported higher direct image preserves quasi-constructibility. 
Though the analogous statement for morphisms with higher dimensional target fails in general, we prove that, in the algebraizable case, it holds after replacing the target with a modification. 
We deduce it from a known finiteness result in the theory of nearby cycles over general bases and a new comparison result, which gives an identification of the compactly supported higher direct image sheaves, up to modification of the target, in terms of nearby cycles over general bases. 
\end{abstract}

\section{Introduction}
We prove a new finiteness theorem and a new comparison theorem on \'etale cohomology of rigid analytic varieties, more specifically, on the compactly supported higher direct image sheaves for algebraizable morphisms of algebraizable rigid analytic varieties. 

\subsection{Finiteness and modification}
Let $k$ be a non-archimedean field and $k^\circ$ denote the subring of power pounded elements of $k$, i.e, the ring of integers. 
Let $\msf f\colon \msf X\to \msf Y$ be a quasi-compact separated morphism of adic spaces locally of finite type over $\Spa(k,k^\circ)$ and $\msf F$ be a quasi-constructible sheaf of $\Z/n\Z$-modules on $\msf X_\et$, with $n$ an integer invertible in $k^\circ$. 
A finiteness theorem of Huber (\cite[Theorem 2.1]{Hub98}) states that, if the characteristic of $k$ is zero and if $\dim\msf Y\le 1$, then the compactly supported higher direct image sheaves $R^i\msf f_!\msf F$ are again quasi-constructible. 
Here, roughly speaking, a quasi-constructible sheaf is a sheaf which \'etale locally admits a stratification by intersections of constructible locally closed subsets and Zariski locally closed subsets such that the restriction to each stratum is locally constant. 
(See \cite{Hub98b}, \cite{Hub07} for variants of this finiteness result.)

We note that the analogous statement for $\msf f$ with $\dim\msf Y\ge2$ fails as pointed out in \cite[Example 2.2]{Hub98}. 
Nonetheless, the author expects that there exists a modification $\pi\colon \msf Y'\to \msf Y$ such that the pullback $\pi^*R^i\msf f_!\msf F$ is quasi-constructible, and we prove this expectation in the case where $\msf f\colon \msf X\to \msf Y$ and $\msf F$ are algebraizable. 

To be precise, let us introduce some notations: 
For a scheme $X$ locally of finite type over $k^\circ$, let $X^\rig$ denote the Raynaud generic fiber of the formal completion of $X$ along the closed fiber, viewed as an adic space. 
For a morphism $f\colon X\to Y$ of schemes locally of finite type over $k^\circ$, let $f^\rig$ denote the associated morphism $X^\rig\to Y^\rig$. 
Finally, for a sheaf $\mc F$ on the generic fiber $X_{\eta,\et}$ of a scheme $X$ locally of finite type over $k^\circ$, let $\mc F^\rig$ denote the sheaf on $X^\rig_\et$ obtained by pulling back $\mc F$. 

Now we can state our finiteness result:
\begin{thm}[Theorem \ref{modfin}]\label{blfin}
Let $k$ be a non-archimedean field and $n$ be an integer invertible in $k^\circ$. 
Let $f\colon X\to Y$ be a morphism of schemes of finite type over $k^\circ$ which is separated and of finite presentation. 
Let $\mc F$ be a constructible sheaf of $\Z/n\Z$-modules on the generic fiber $X_{\eta,\et}$. 
Then there exists a modification $\pi\colon Y'\to Y$ such that the sheaves $\pi^{\rig*}R^if^\rig_!\mc F^\rig$ are quasi-constructible. 
\end{thm}

Our proof is totally different from Huber's and reproves his finiteness result in the algebraizable case. 
Contrary to Huber's case, we do not need any assumption on the characteristic of $k$. 
In fact, we will not make any reference to the base field in the proof and will work with much more general setting (see \S\ref{sec blfin} for details). 

Further, our approach gives a remarkable simplification of the proof. 
We do not need any of complicated reduction arguments that Huber made in \cite{Hub98}. 
We do not rely on deep results used in his proof, namely, the preservation of constructibility by smooth compactly supported pushforward \cite[Theorem 6.2.2]{Hub96}, and the $p$-adic Riemann's existence theorem due to L\"utkebohmert \cite[Theorem 2.2]{Lut93}, which requires the assumption on the characteristic of $k$. 
We prove Theorem \ref{blfin} as a direct consequence of the following two results: 
One is a known finiteness result in the theory of nearby cycles over general bases due to Orgogozo \cite{Org06}, and the other is a new comparison result, which gives an identification of the compactly supported higher direct image sheaves, up to modification of the target, in terms of nearby cycles over general bases. 
This comparison result is the heart of this paper and will be briefly explained in the next subsection. 

The idea of our approach comes from a recent work of K.\ Ito \cite{Ito20}, in which the theory of nearby cycles over general bases is used to study \'etale cohomology of rigid analytic varieties. 
Though his aim is different from proving a finiteness result, his work includes a new proof of a special case of Huber's finiteness result; namely the special case of Theorem \ref{blfin} where $\dim Y^\rig=1$ and $\mc F$ is constant (\cite[Theorem 6.10]{Ito20}). 
Let us mention that his proof does not rely on the $p$-adic Riemann existence theorem, but still relies on the preservation of constructibility by smooth compactly supported pushforward. 

It is natural to ask if one can also apply our argument to higher direct image sheaves without support. 
The author expects that a similar statement with ``quasi-constructible'' replaced by ``oc-quasi-constructible'' (see \cite{Hub98b} for the definition and the case where $\dim \msf Y=1$) holds. 
However, it seems much more difficult and he does not know how to address it even in the algebraizable case (see Remark \ref{not qc}).  

One of typical applications of a finiteness result of higher direct image is the existence of a tubular neighborhood that does not change the cohomology. 
For instance, in \cite{Hub98}, the following is proved as a consequence of his finiteness result: 
Under the assumption that $k$ is of characteristic zero, every hypersurface in an affinoid adic space of finite type over $\Spa(k,k^\circ)$ admits a tubular (closed) neighborhood that does not change the compactly supported cohomology (see \cite{Hub98}, \cite{Hub98b}, \cite{Ito20} for variants). 
In a celebrated paper of Scholze \cite{Sch12}, such a existence result is used in the proof of the weight monodromy conjecture for varieties of complete intersection. 

Similarly, as an application of Theorem \ref{blfin}, we give a variant for families of rigid analytic varieties. 
For instance, we prove the following. 
We consider an algebraizable family of hypersurfaces parametrized by the unit disk, i.e, $\msf X\subset \msf \Spa(A,A^\circ)\to \B^1_k=\Spa(k\la T\ra,k^\circ\la T\ra)$ with $A$ topologically of finite type over $k\la T\ra$ and $\msf X$ a hypersurface defined by some $f\in A$.  
We show that there exist a small neighborhood $\B^1_k(\delta)=\{x\in\B^1_k\mid |T(x)|\le\delta\}$ of the origin of $\B^1_k$ with $\delta\in|k^\times|$ and a tubular (closed) neighborhood $\msf X$ defined over the punctured disk $\B^1_k(\delta)^*=\{x\in\B^1_k\mid 0<|T(x)|\le\delta\}$ which has the same compactly supported higher direct image sheaves on $\B^1_k(\delta)^*$.  
For the precise statement, see Proposition \ref{tube}.
This result might lead us to a relative version of Scholze's approach to the weight monodromy conjecture for varieties of complete intersection. 

\subsection{Comparison result}\label{intro comparison}
One of the most fundamental results in the theory of \'etale cohomology of rigid analytic varieties is the comparison theorem between the \'etale cohomology of an algebraizable rigid analytic variety and the nearby cycle cohomology, due to Fujiwara and Huber (\cite[Corollary 6.6.2]{Fuj95}, \cite[Theorem 5.7.6]{Hub96}). 
More generally, the results in Huber's book \cite{Hub96} show the following description of the stalks of the compactly supported higher direct image sheaves: 
Let $k$ be a non-archimedean field. Let $f\colon X\to Y$ be a separated morphism of schemes of finite type over $k^\circ$ and $\mc F$ a sheaf of $\Z/n\Z$-modules on the generic fiber $X_{\eta,\et}$, for an integer $n$.  
We consider the higher direct image $Rf^\rig_!\mc F^\rig$ by the morphism $f^\rig\colon X^\rig\to Y^\rig$ between the Raynaud generic fibers. 
For every geometric point $\xi$ (possibly of higher rank) of $Y^\rig$, the stalk $(Rf^\rig_!\mc F^\rig)_\xi$ is canonically isomorphic to the cohomology of the nearby cycle taken over the valuation ring $\kappa(\xi)^+$ corresponding to $\xi$. 

Thus, we know that there is a comparison isomorphism at each geometric point between the stalk of $Rf^\rig_!\mc F^\rig$ and a suitable nearby cycle cohomology. 
The main discovery of this article is a uniform construction of those separately constructed isomorphisms. 
More precisely, we construct a relative variant of nearby cycle cohomology as a complex on $Y^\rig_\et$ and a canonical morphism to $Rf^\rig_!\mc F$. 
Then we prove that the canonical morphism becomes an isomorphism after a modification of $Y$, under the assumption that $n$ is invertible in $k^\circ$ and that $\mc F$ is constructible. 

For the construction, we use the theory of nearby cycles over general bases as in \cite{Ill14}. 
We first construct a relative variant of nearby cycle cohomology as a complex $\Xi_f\mc F$ on the topos $Y\vani_YY$ called the vanishing topos of $Y$ (see \S\ref{c push} for details). 
Roughly speaking, the vanishing topos $Y\vani_YY$ is a topos whose points correspond to specializations $y\gets z$ of two geometric points of $Y$, and the complex $\Xi_f\mc F$ is a complex whose stalk at $y\gets z$ is canonically isomorphic to the cohomology of the nearby cycle taken with respect to the specialization $y\gets z$ over the strict henselization $Y_{(y)}$. 
Here, as mentioned in Illusie's survey article \cite[6.4]{Ill06}, the vanishing topos $Y\vani_YY$ receives a canonical morphism of topoi $\til\lambda_Y:(Y^\rig_\et)^\sim\to Y\vani_YY$ from the \'etale topos of $Y^\rig$ which sends a geometric point $\xi$ of $Y^\rig$ to the specialization defined by the closed point and the generic point of the spectrum of the valuation ring $\kappa(\xi)^+$. 
Then the desired complex on $Y^\rig_\et$ is defined to be the pullback $\til\lambda^*_Y\Xi_f\mc F$. 

\begin{thm}\label{construction}
Let $f\colon X\to Y$ be a compactifiable morphism of schemes of finite type over $k^\circ$ and $\mc F$ be a sheaf of $\Z/n\Z$-modules on the generic fiber $X_{\eta,\et}$, for an integer $n$. 
Then there exists a canonical morphism 
\be
\theta_f\colon \til\lambda_Y^*\Xi_f\mc F\to Rf^\rig_!\mc F^\rig
\ee
satisfying the following: 
Assume that $n$ is invertible in $k^\circ$, that $f$ is of finite presentation, and that $\mc F$ is constructible. 
Then, after replacing $Y$ by a modification and $f$ by the base change, the morphism $\theta_f$ becomes an isomorphism. 
\end{thm}

We note that the nearby cycle appearing in the description of $(Rf^\rig_!\mc F^\rig)_\xi$ is taken over the valuation ring $\kappa(\xi)^+$, while that appearing in $(\til\lambda_Y^*\Xi_f\mc F)_\xi$ is taken over the strict henselization. 
In general, the formation of nearby cycle does not commute with base change, and in particular, the above two nearby cycles are different. 
But, as proved by Orgogozo in \cite{Org06}, it commutes with base change after a modification, i.e, after replacing $Y$ by a modification and $f$ by the base change, the formation of the nearby cycle commutes with base change. In particular, the above two nearby cycles agree after this replacement. 

Once we have proved Theorem \ref{construction}, we can easily deduce Theorem \ref{blfin} from the constructibility of the complex $\Xi_f\mc F$ on the vanishing topos $Y\vani_YY$, which is a consequence of finiteness results on nearby cycles over general bases due to Orgogozo \cite{Org06}. 

The canonical morphism $\theta_f$ in Theorem \ref{construction} is constructed as a certain base change morphism with respect to a certain square diagram involving the vanishing topoi of $X$ and $Y$ (see Construction \ref{rel comparison}). 
To check that $\theta_f$ becomes isomorphic after modification, we show that, using results of Orgogozo in \cite{Org06}, after replacing $Y$ by a modification and $f$ by the base change, the formation of $\Xi_f\mc F$ commutes with base change. 
After this replacement, the morphism $(\theta_f)_\xi$ induced on the stalks at each geometric point $\xi$ of $Y^\rig$ is identified with the comparison morphism constructed by Huber.

\subsection{Organization of the article}
We recall basic definitions and some results on the theory of nearby cycles over general bases in \S\ref{sec vani}.
After a preliminary section \S\ref{sec adic} concerning on adic spaces, 
we give in \S\ref{sec rigid} the key constructions and prove Theorem \ref{construction}. 
We deduce Theorem \ref{blfin} from Theorem \ref{construction} in \S\ref{sec blfin}. 
Finally we deduce the existence of a good tubular neighborhood for families in \S\ref{sec tube}.

\paragraph{Acknowledgments}
The author would like to thank Teruhisa Koshikawa for pointing out mistakes in an earlier attempt of the author to directly generalize the proof in \cite{Hub98} and for suggesting to use nearby cycles over general bases. 
He is grateful to Kazuhiro Ito for explaining his work \cite{Ito20} to the author, especially how to apply nearby cycles over general bases to the theory of \'etale cohomology of rigid analytic varieties. 
He would also like to thank Daichi Takeuchi for various helpful discussions, especially for pointing out a mistake in Remark \ref{punc} of the earlier draft and suggesting to fix Definition \ref{adm} and Definition \ref{val triple}. 
Finally, he would like to express his gratitude to Takeshi Saito for various helpful discussion and for constant encouragement. 

This project has received funding from Iwanami Fujukai Foundation and the European Research Council (ERC) under the European Union's Horizon 2020 research and innovation programme (grant agreement No.\ 851146).

\paragraph{Convention}
For a scheme, or a pseudo-adic space $X$, we denote the associated \'etale site by $X_\et$. 
We use the same letter $X$ for the associated \'etale topos by abuse of notation. 

A morphism $f\colon X\to Y$ of schemes is called {\it compactifiable} if there exist a proper morphism $\bar f\colon \bar X\to Y$ and an open immersion $j\colon X\to \bar X$ such that $f=\bar fj$. 
Note that, by Nagata's compactification theorem (\cite[Chapter II, Theorem F.1.1]{FK18}), if $f$ is separated and of finite type and if $Y$ is quasi-compact and quasi-separated, then $f$ is compactifiable.  

\section{Vanishing topos and Nearby cycles over general bases}
\label{sec vani}
\subsection{Review of definition}
Let $f\colon X\to S$ and $g\colon Y\to S$ be morphisms of topoi. 
For the definition of the oriented product $X\vani_SY$, we refer to \cite{Ill14}. 
It comes with morphisms $p_1\colon X\vani_SY\to X$ and $p_2\colon X\vani_SY\to Y$ and a $2$-morphism $\tau\colon gp_2\to fp_1$, and is universal for these data. 
Recall that giving a $2$-morphism $a\to b$ between two morphisms $a,b\colon X\to Y$ of topoi is equivalent to giving a morphism $a_*\to b_*$ of functors (\cite[IV.3.2]{SGA4}).

If $f\colon X\to S$ and $g\colon Y\to S$ are morphisms of schemes, then a point of the topos $X\vani_SY$, i.e., a morphism from the punctual topos, is described by a triple $(x,y,c)$ consisting of a geometric point $x$ of $X$, a geometric point $y$ of $Y$, and a specialization $c\colon g(y)\to f(x)$ of geometric points, i.e, a morphism $g(y)\to S_{(f(x))}$ of schemes. 
We use the notation $x\gets y$ for the triple $(x,y,c)$. 

For a morphism $f\colon X\to S$ of schemes, the topos $X\vani_SS$ is called the vanishing topos and we have the natural morphism $\Psi_f\colon X\to X\vani_SS$ induced from $\id_X\colon X\to X$, $f\colon X\to S$, and the $2$-morphism $\id\colon f\to f$ by the universal property of $X\vani_SS$. 
For a ring $\Lambda$, the nearby cycle functor 
$R\Psi_f\colon D^+(X,\Lambda)\to D^+(X\vani_YY,\Lambda)$ 
is defined to be the derived functor $R(\Psi_f)_*$.

\begin{rem}\label{p2comp}
Let $X$ be a scheme and $\Lambda$ be a ring. 
We consider the $2$-commutative diagram
\bx{X\ar[r]^-{\Psi_\id}\ar[rd]_-\id
&X\vani_XX\ar[d]^-{p_2}
\\&X.}\ex
Recall that the morphism $p_2^*\to R\Psi_\id$ of functors $D^+(X,\Lambda)\to D^+(X\vani_XX,\Lambda)$ induced by $p_2\Psi_\id\cong\id$ is an isomorphism \cite[Proposition 4.7]{Ill14}. 

For a morphism $f\colon X\to S$ of schemes, we have the $2$-commutative diagram 
\bx{X\ar[r]^-{\Psi_\id}\ar[rd]_{\Psi_f}&X\vani_XX\ar[d]^{\id\vani f}\\&X\vani_SS.}\ex
Thus, we have canonical isomorphisms 
\be
R\Psi_f\cong R(\id\vani f)_*R\Psi_\id\cong R(\id\vani f)_*p_2^*.
\ee
\end{rem}

\subsection{Base change morphisms}\label{sec bc}
We recall the base change morphism for $R\Psi$. 
For a morphism $f\colon X\to Y$ of schemes, we consider base change by a morphism $Y^\prime\to Y$. Consider the Cartesian diagram
\bx{X'\ar[r]^{g}\ar[d]_{f'}&X\ar[d]^f\\Y'\ar[r]&Y}\ex
of schemes, which induces the diagram of topoi
\bx{X'\ar[r]^{g}\ar[d]_{\Psi_{f'}}&X\ar[d]^{\Psi_f}\\X'\vani_{Y'}Y'\ar[r]^{\ola g}&X\vani_YY.}\ex
We consider the base change morphism 
\be
{\ola g}^*R\Psi_{f}\to R\Psi_{f'}g^*. 
\ee

\begin{defn}
Let $f\colon X\to Y$ be a morphism of schemes, $\Lambda$ a ring,  and $\mc F$ an object of $D^+(X,\Lambda)$. 
We say that the formation of $R\Psi_f\mc F$ commutes with base change if, 
for every morphism $Y'\to Y$ of schemes, the base change morphism 
\be
{\ola g}^*R\Psi_{f}\mc F\to R\Psi_{f'}g^*\mc F
\ee
defined above is an isomorphism. 
\end{defn}

The following analogue of the proper base change theorem is indispensable in our construction. 
\begin{lem}\label{p1push}
Let 
\bx{
X'\ar[r]\ar[d]^{f'}&X\ar[d]^f\\Y'\ar[d]\ar[r]&Y\ar[d]\\S'\ar[r]&S
}\ex
be a commutative diagram of schemes with the upper square being Cartesian. 
We assume that $f$ is proper. 
We consider the diagram of topoi 
\bx{
X'\vani_{S'}S'\ar[d]^{f'\vani\id}\ar[r]^{\varphi'}&X\vani_SS\ar[d]^{f\vani\id}\\
Y'\vani_{S'}S'\ar[r]^{\varphi}&Y\vani_SS
}\ex
Let $\Lambda$ be a torsion ring.
Then the base change morphism 
\be
\varphi^*R(f\vani\id)_*\to R(f'\vani\id)_*{\varphi'}^*
\ee
is an isomorphism of functors $D^+(X\vani_SS,\Lambda)\to D^+(Y'\vani_{S'}S',\Lambda)$. 
\end{lem}

Though this is essentially proved in \cite[Lemme 10.1]{Org06}, 
we include a proof for completeness. 

Recall that, if $X\to S$ and $Y\to S$ are coherent morphisms of coherent topoi, then the oriented product $X\vani_SY$ is also coherent \cite[Lemma 2.5]{Ill14}. 
Thus, by \cite[Proposition VI.9.0]{SGA4}, it has enough points, that is, a morphism $\mc F\to \mc G$ of objects of $X\vani_SY$ is an isomorphism if and only if the morphism $\mc F_{x\gets y}\to \mc G_{x\gets y}$ is bijective for any point $x\gets y$ of $X\vani_SY$ \cite[D\'efinition IV.6.4.1]{SGA4}. 
This implies that for any diagram $X\to S\gets Y$ of schemes, the oriented product $X\vani_SY$ has enough points. 

\begin{proof}[Proof of Lemma \ref{p1push}]
For each object $\mc F$ of $D^+(X\vani_SS,\Lambda)$ and for each point $y'\gets t'$ of $Y'\vani_{S'}S'$, we check that the morphism induced on the stalks is an isomorphism. 
For this, we may assume that $Y'=Y'_{(y')}$ and $Y=Y_{(y)}$, where $y$ is the image of $y'$. 
Further, by \cite[1.12]{Ill14}, we may assume that $S'=S'_{(s')}$ and $S=S_{(s)}$ are strictly local, where $s'$ and $s$ are the images of $y'$ in $S'$ and $S$ respectively. 
We denote the image of $t'$ in $S$ by $t$. 
Then the morphism $(R(f\vani\id)_*\mc F)_{y\gets t}\to (R(f'\vani\id)_*{\varphi'}^*\mc F)_{y'\gets t'}$ induced on the stalks is identified with the morphism
\be
R\Gamma(X\vani_SS_{(t)},\mc F)\to R\Gamma(X'\vani_{S'}S'_{(t')},\mc F)
\ee
induced by the adjunction morphism $\id\to R\varphi'_*{\varphi'}^*$. 
We consider the diagram
\bx{
X'\vani_{S'}S'_{(t')}\ar[r]^{\varphi'}\ar[d]^{p'_1}
&X\vani_SS_{(t)}\ar[d]^{p_1}
\\X'\ar[r]^{g}
&X.
}\ex
The morphism in question is identified with
\be
R\Gamma(X,Rp_{1*}\mc F)\to R\Gamma(X',Rp'_{1*}{\varphi'}^*\mc F).
\ee
Then, by the (usual) proper base change theorem, it suffices to show that the base change morphism $g^*Rp_{1*}\mc F\to Rp'_{1*}{\varphi'}^*\mc F$ is an isomorphism. 
We check that, for each geometric point $x'$ of $X'$, the morphism $(Rp_{1*}\mc F)_x\to (Rp'_{1*}{\varphi'}^*\mc F)_{x'}$ induced on the stalks is an isomorphism, where $x$ is the image of $x'$. 
But it is identified with the natural morphism 
$R\Gamma(X_{(x)}\vani_SS_{(t)},\mc F)\to R\Gamma(X'_{(x')}\vani_{S'}S'_{(t')},{\varphi'}^*\mc F)$, which is an isomorphism since $\varphi'\colon X'_{(x')}\vani_{S'}S'_{(t')}\to X_{(x)}\vani_SS_{(t)}$ is a local morphism of local topoi \cite[Corollaire 2.3.2]{Ill14}. 
Thus, the assertion follows. 
\end{proof}

\subsection{Higher direct image with compact support for vanishing topoi}
\label{c push}
Let $S$ be a scheme and $f\colon X\to Y$ be a morphism of $S$-schemes. 
We consider the natural morphism 
\be
f\vani \id\colon X\vani_SS\to Y\vani_SS.
\ee
We assume that $f\colon X\to Y$ is compactifiable and choose a proper morphism $\bar f\colon \bar X\to Y$ with an open immersion $j\colon X\to \bar X$ such that $\bar fj=f$.  
For a torsion ring $\Lambda$, we define a functor
\be
R(f\vani \id)_!\colon D^+(X\vani_SS,\Lambda)\to D^+(Y\vani_SS,\Lambda)
\ee
to be the composite $R(\bar f\vani\id)_*\circ (j\vani\id)_!$ (cf. \cite[Construction 1.8]{LZ19}). 
Note that $j\vani\id\colon X\vani_SS\to \bar X\vani_SS$ is an open immersion of topoi. 
Lemma \ref{p1push} (together with \cite[XVII.3.3]{SGA4}) implies 
that the composite $R(\bar f\vani\id)_*\circ (j\vani\id)_!$ is independent of the choice of a compactification up to a canonical isomorphism. 
Lemma \ref{p1push} also implies the following:
\begin{lem}\label{p1cpush}
Let notation be as in Lemma \ref{p1push}. 
We assume that $f$ is compactifiable. 
Then we have a canonical isomorphism 
\be
\varphi^*R(f\vani\id)_!\cong R(f'\vani\id)_!{\varphi'}^*
\ee
of functors $D^+(X\vani_SS,\Lambda)\to D^+(Y'\vani_{S'}S',\Lambda)$. 
\end{lem}

We consider a pair $(X,L_X)$ of a scheme $X$ and a constructible closed subset $L_X$ of $X$. 
We call such a pair a scheme with support. 
A morphism $f\colon (X,L_X)\to (Y,L_Y)$ of schemes with support is a morphism $f\colon X\to Y$ of schemes such that $L_X\subset f^{-1}(L_Y)$. 
We often regard $L_X$ as a subscheme of $X$ with the reduced scheme structure. 

Let $f\colon (X,L_X)\to (Y,L_Y)$ be a morphism of schemes with support. 
Let $\Lambda$ be a torsion ring and $\mc F$ be an object of $D^+(X,\Lambda)$. 
We are interested in the complex $R(f\vani\id)_!R\Psi_f\mc F$ on $L_Y\vani_YY$, where $f\vani\id$ denotes the natural morphism $L_X\vani_YY\to L_Y\vani_YY$, and where the restriction of $R\Psi_f\mc F$ is also denoted by $R\Psi_f\mc F$. 
This complex plays a role of a family of nearby cycle cohomology, which was denoted by $\Xi_f\mc F$ in the introduction. 

\begin{const}\label{bc xi}
Let $\mc F$ be an object of $D^+(X,\Lambda)$. 
We study a base change morphism for the complex $R(f\vani\id)_!R\Psi_f\mc F$. 
Let 
\bx{(X',L_{X'})\ar[r]^{g'}\ar[d]^{f'}
&(X,L_X)\ar[d]^f
\\(Y',L_{Y'})\ar[r]^{g}
&(Y,L_Y)}\ex
be a Cartesian diagram of schemes with support (being Cartesian means that the underlying diagram of schemes is Cartesian and ${f'}^{-1}(L_{Y'})\cap {g'}^{-1}(L_X)=L_{X'}$). 
Recall that we have the base change morphism 
\be
\ola{g'}^*R\Psi_f\mc F\to R\Psi_{f'}({g'}^*\mc F),
\ee
where $\ola{g'}$ is the natural morphism $L_{X'}\vani_{Y'}Y'\to L_{X}\vani_YY$ (see the beginning of \S\ref{sec bc}). 
Combining this with the base change isomorphism in Lemma \ref{p1cpush}, we obtain
\be
\alpha\colon \ola{g}^*R(f\vani\id)_!R\Psi_f\mc F&\cong &R(f'\vani\id)_!\ola{g'}^*R\Psi_{f}\mc F\\&\to &R(f'\vani\id)_!R\Psi_{f'}{g'}^*\mc F,
\ee
where $\ola g$ denotes the natural morphism $L_{Y'}\vani_{Y'}Y'\to L_Y\vani_YY$. 
\end{const}

It will be convenient to prepare the following notation. 

\begin{defn}\label{full c push}
For a morphism $f\colon (X,L_X)\to (Y,L_Y)$ of schemes with support with $f\colon X\to Y$ being compactifiable, we consider the natural morphism $\ola f\colon L_X\vani_XX\to L_Y\vani_YY$ of topoi. 
We define a functor 
\be
R\ola f_!\colon D^+(L_X\vani_XX,\Lambda)\to D^+(L_Y\vani_YY,\Lambda)
\ee 
by $R\ola f_!=R(f\vani\id)_!R(\id\vani f)_*$. 
\end{defn}

\begin{lem}\label{j and push}
Let $f\colon (X,L_X)\to (Y,L_Y)$  be a morphism of schemes with support such that $f\colon X\to Y$ is compactifiable. 
We take morphisms $(X,L_X)\overset{j}{\to}(\bar X,L_{\bar X})\overset{\bar f}{\to} (Y,L_Y)$ of schemes with support such that $\bar f\colon \bar X\to Y$ is a proper morphism, $j\colon X\to \bar X$ is an open immersion, and $\bar fj=f$. 
Then we have a canonical isomorphism 
\be
R\ola f_!\cong R\ola{\bar f}_*\ola j_!,
\ee
where $\ola{\bar f}$ denote the natural morphism $L_{\bar X}\vani_{\bar X}\bar X\to L_Y\vani_YY$ and $\ola j$ the natural open immersion $L_X\vani_XX=L_X\vani_{\bar X}\bar X\to L_{\bar X}\vani_{\bar X}\bar X$. 
\end{lem}

\begin{proof}
Consider the diagram of topoi
\bx{
L_{X}\vani_XX\ar[r]^{j\vani \id_X}\ar[d]^{\id\vani f}
&L_{\bar X}\vani_{\bar X}\bar X\ar[d]^{\id \vani \bar f}
\\L_X\vani_YY\ar[r]^{j\vani \id_Y}
&L_{\bar X}\vani_YY,
}\ex
which induces a natural isomorphism 
\be
(j\vani\id_Y)^*R(\id\vani \bar f)_*\to R(\id \vani f)_*(j\vani\id_X)^*.
\ee 
Its inverse induces a morphism 
\be
R(\id\vani f)_*=R(\id\vani f)_*(j\vani\id_X)^*(j\vani\id_X)_!\to (j\vani\id_Y)^*R(\id \vani\bar f)_*(j\vani\id_X)_!.
\ee
By adjunction, we obtain
\ben\label{jj}
(j\vani\id_Y)_!R(\id\vani f)_*\to R(\id \vani\bar f)_*(j\vani\id_X)_!.
\een
Now let $i$ denote the closed immersion $Z=L_{\bar X}\setminus L_X\to L_{\bar X}$. 
Then the base change morphism 
\be
(i\vani\id_{Y})^*R(\id_{L_{\bar X}}\vani \bar f)_*\to R(\id_{{Z}}\vani \bar f)_*(i\vani\id_{\bar X})^*
\ee
is an isomorphism by \cite[Corollary 1.2]{LZ19}. 
Thus, the morphism (\ref{jj}) is an isomorphism, and hence, we obtain a canonical isomorphism $R\ola f_!\cong R\ola{\bar f}_*\ola j_!$. 
\end{proof}

Let $f\colon (X,L_X)\to (Y,L_Y)$ be a morphism of schemes with support. 
We consider the following morphisms of topoi:
\bx{
X&L_X\vani_XX\ar[l]_-{p_2}\ar[r]^{\id\vani f}&L_X\vani_YY\ar[r]^{f\vani\id}&L_Y\vani_YY.
}\ex
We note that, by Remark \ref{p2comp}, we have a canonical isomorphism $R\ola f_!p_2^*\cong R(f\vani\id)_!R\Psi_f$. 

The following will be used (only) in the proof of Lemma \ref{comm w bc}. 

\begin{const}\label{bc bc}
Again we consider a Cartesian diagram 
\bx{(X',L_{X'})\ar[r]^{g'}\ar[d]^{f'}
&(X,L_X)\ar[d]^f
\\(Y',L_{Y'})\ar[r]^{g}
&(Y,L_Y)}\ex
of schemes with support. 
We consider the diagram
\bx{
L_{X'}\vani_{X'}X'\ar[r]^{\ola g''}\ar[d]^{\id\vani f'}
&L_X\vani_XX\ar[d]^{\id\vani f}\ar[r]^-{p_2}
&X
\\L_{X'}\vani_{Y'}Y'\ar[r]^{\ola g'}\ar[d]^{f'\vani \id}
&L_X\vani_YY\ar[d]^{f\vani\id}
\\L_{Y'}\vani_{Y'}Y'\ar[r]^{\ola g}
&L_Y\vani_YY.}\ex
This induces a base change morphism
\be
\ola g^*R\ola f_!\to R\ola{f'}_!\ola{g''}^*
\ee
of functors $D^+(L_X\vani_XX,\Lambda)\to D^+(L_{Y'}\vani_{Y'}Y',\Lambda)$. 
For an object $\mc F$ of $D^+(X,\Lambda)$, we have the following commutative diagram
\bx{
\ola g^*R(f\vani\id)_!R\Psi_f\mc F\ar[r]^\alpha\ar[d]
&R(f'\vani\id)_!R\Psi_{f'}({g'}^*\mc F)\ar[d]
\\\ola g^*R\ola f_!p_2^*\mc F\ar[r]
&R\ola{f'}_!\ola{g''}^*p_2^*\mc F.
}\ex
in $D^+(L_{Y'}\vani_{Y'}Y',\Lambda)$.
\end{const}

\subsection{Constructibility}
For a scheme or a morphism of schemes, being coherent means being quasi-compact and quasi-separated. 

\begin{defn}[{\cite[\S\S8--9]{Org06}}]
Let $X\to S$ and $Y\to S$ be coherent morphisms of coherent schemes. 
Let $\Lambda$ be a noetherian ring. 
We say that a sheaf $\mc F$ of $\Lambda$-modules on $X\vani_SY$ is {\it constructible }if there exist partitions $X=\coprod_{i\in I}X_i$ and $Y=\coprod_{j\in J}Y_i$ by finitely many locally closed constructible subsets $X_i\subset X$ and $Y_j\subset Y$ such that the restriction of $\mc F$ to the subtopos $X_i\vani_SY_j$ is locally constant of finite type for every $(i,j)\in I\times J$. 
\end{defn}

The full subcategory of the category of sheaves of $\Lambda$-modules on $X\vani_SY$ consisting of constructible sheaves is a thick subcategory, i.e., closed under extension. 
Thus, complexes $\mc K$ such that the cohomology sheaves $\mc H^i(\mc K)$ are constructible for all $i$ form a triagulated subcategory of $D^b(X\vani_SY,\Lambda)$, 
which is denoted by $D^b_c(X\vani_SY,\Lambda)$.  

\begin{rem}\label{psi good implies c}
Let $f\colon X\to Y$ be a morphism of finite presentation of coherent schemes with $Y$ having only finitely many irreducible components. 
Let $\mc F$ be an object of $D^b_c(X,\Z/n\Z)$, for an integer $n$ invertible on $Y$. 
Then, by \cite[Th\'eor\`eme 8.1, Lemme 10.5]{Org06}, if the formation of $R\Psi_f\mc F$ commutes with base change, then $R\Psi_f\mc F$ is an object of $D_c^b(X\vani_YY,\Z/n\Z)$. 
\end{rem}

\begin{lem}\label{c of family of psi}
Let $f\colon (X,L_X)\to (Y,L_Y)$ be a morphism of schemes with support (\S\ref{c push}) such that $Y$ is a coherent scheme having only finitely many irreducible components and $f\colon X\to Y$ is a separated morphism of finite presentation. 
Let $n$ be an integer invertible on $Y$ and $\mc F$ an object of $D^b_c(X,\Z/n\Z)$. 
We assume that the formation of $R\Psi_f\mc F$ commutes with base change. 
Then the complex $R(f\vani\id)_!R\Psi_f\mc F$ is an object of $D^b_c(L_Y\vani_YY,\Z/n\Z)$, where $f\vani\id$ denotes the natural morphism $L_X\vani_YY\to L_Y\vani_YY$ and $R\Psi_f\mc F$ also denotes its restriction to $L_X\vani_YY$. 
\end{lem}
\begin{proof}
By Remark \ref{psi good implies c}, the complex $R\Psi_f\mc F$ is an object of $D^b_c(X\vani_YY,\Lambda)$. 
Then the assertion follows from \cite[Proposition 10.2]{Org06} together with a standard limit argument. 
\end{proof}

\section{Preliminaries on analytic adic spaces}
\label{sec adic}
We refer to \cite{Hub94}, \cite{Hub96} for details on adic spaces. 
\subsection{Analytic adic spaces}
We first recall from \cite[1.9]{Hub96} the analytic adic space associated to a formal scheme. 
We consider a formal scheme $\mf X$ satisfying the following condition from \cite[1.9]{Hub96}: 
\begin{description}
\item{(*)} Any point of $\mf X$ admits an affine neighborhood $\Spf A\subset \mf X$ such that the topology of $A$ is $\varpi$-adic for some element $\varpi\in A$ and the ring $A[1/\varpi]$ equipped with the topology induced from $A$ is a strongly noetherian Tate ring (see \cite[\S2]{Hub94} or \cite[1.1]{Hub96} for the definition of being strongly noetherian). 
\end{description}
Then we can consider the analytic adic space $d(\mf X)$ associated to $\mf X$ defined in \cite[Proposition 1.9.1]{Hub96}; if $\mf X=\Spf A$ is affine with the topology of $A$ being $\varpi$-adic for some $\varpi\in A$, then $d(\mf X)=\Spa(A[1/\varpi],A^+)$, where $A^+$ is the integral closure of $A$ in $A[1/\varpi]$. 
The adic space $d(\mf X)$ comes with a natural morphism $\lambda_{\mf X}\colon d(\mf X)\to \mf X$ of locally topologically ringed spaces. 
It induces a natural morphism $d(\mf X)_\et\to \mf X_\et$ of sites as in \cite[Lemma 3.5.1]{Hub96}, which we denote again by $\lambda_{\mf X}$ and call the specialization morphism. 
In this article we focus on formal schemes coming from schemes.  

\begin{defn}\label{adm}
Let $X$ be a scheme and $X_0$ the closed subset defined by vanishing of a locally finitely generated ideal sheaf. 
Consider the formal completion $\wh X$ of $X$ along a locally finitely generated ideal sheaf defining $X_0$ (note that $\wh X$ is independent of the choice of a locally finitely generated ideal sheaf). 
The pair $(X,X_0)$ is called an {\it admissible pair} if $\wh X$ satisfies the condition (*). 
A morphism $(X,X_0)\to (Y,Y_0)$ of admissible pairs is a morphism $f\colon X\to Y$ of schemes such that $X_0=f^{-1}(Y_0)$. 
\end{defn}

For an admissible pair $(X,X_0)$, let $\wh X$ denote the formal completion of $X$ along a locally finitely generated ideal sheaf defining $X_0$ and let $(X,X_0)^a$ be the associated analytic adic space $d(\wh X)$.  
The assignment $(X,X_0)\mapsto (X,X_0)^a$ defines a functor from the category of admissible pairs to that of analytic adic spaces. 
If there is no risk of confusion, we simply write $X^a$ for $(X,X_0)^a$. 

Let $U$ denote the complementary open $X\setminus X_0$. 
The adic space $X^a$ comes with the following two natural morphisms of sites;
\begin{itemize}
\item $\lambda_X\colon X^a_\et\to X_{0,\et}$ defined to be the composite of the specialization morphism $\lambda_{\wh X}\colon X^a_\et\to \wh X_\et$ and the natural equivalence $\wh X_\et\cong X_{0,\et}$,
\item $\varphi_X\colon X^a_\et\to U_\et$ defined in \cite[3.5.12]{Hub96}, which we call the analytification morphism.
\end{itemize}

\begin{exa}\label{rigid}
Let $k$ be a non-archimedean field and $X$ a scheme locally of finite type over $k^\circ$. 
Then we can form an admissible pair $(X,X_0)$ by letting $X_0$ be the closed fiber of $X$. 
Then the analytic adic space $(X,X_0)^a$ is nothing but the Raynaud generic fiber (\cite{Ray74}, \cite[Remark 4.6.ii]{Hub94}, \cite[Example 1.9.2.ii]{Hub96}) viewed as an adic space, which is denoted by $X^\rig$. 
Similarly, for a morphism $f\colon X\to Y$, we denote the associated morphism $X^\rig\to Y^\rig$ by $f^\rig$. 
For a sheaf $\mc F$ on $X_\et$, we denote the pullback $\varphi_X^*\mc F$ by $\mc F^\rig$. 
\end{exa}

\subsection{Analytic pseudo-adic spaces}
In the proof of Theorem \ref{construction}, we will need to consider geometric points of adic spaces, for which we need to work with the framework of pseudo-adic spaces. 
We refer to \cite[1.10]{Hub96} for details on pseudo-adic spaces. 
Let us just recall that a pseudo-adic space is a pair $X=(\underline X,|X|)$ of an adic space $\underline X$ and a subset $|X|$ of $\underline X$ satisfying certain conditions.
A morphism $f\colon (\underline X,|X|)\to (\underline Y,|Y|)$ of pseudo-adic spaces is a morphism $f\colon \underline X\to \underline Y$ of adic spaces such that $f(|X|)\subset|Y|$. 
We often regard an adic space $X$ as a pseudo-adic space via the fully faithful functor $X\mapsto (X,X)$ from the category of adic spaces to that of pseudo-adic spaces. 

As in \cite[3.5.3]{Hub96}, to a pair $(\mf X,L)$ of a formal scheme $\mf X$ satisfying (*) and a locally closed subset $L$ of $\mf X$, we can associate an analytic pseudo-adic space $(d(\mf X),\lambda_{\mf X}^{-1}(L))$, which is denoted by $d(\mf X,L)$. 
It comes with a natural morphism $\lambda_{(\mf X,L)}\colon d(\mf X,L)_\et\to L_\et$ of sites \cite[3.5.3, 3.5.5]{Hub96}, where $L$ is regarded as a subscheme of $\mf X$ with reduced scheme structure. 
We call $\lambda_{(\mf X,L)}$ the specialization morphism. 

We define an algebro-geometric datum which gives an analytic pseudo-adic space. 

\begin{defn}\label{triple}
An {\it admissible triple }is a triple $(X,X_0,L)$ with $(X,X_0)$ being an admissible pair (Definition \ref{adm}) and $L$ being a constructible closed subset of $X_0$. 
A morphism $(X,X_0,L)\to (Y,Y_0,M)$ of admissible triples is a morphism $f\colon (X,X_0)\to (Y,Y_0)$ of admissible pairs such that $f(L)\subset M$. 
\end{defn}
For an admissible triple $(X,X_0,L)$, we can consider the pseudo-adic space $d(\wh X,L)$ associated to the pair $(\wh X,L)$, which we denote by $(X,X_0,L)^a$, or simply by $X^a$. 
The assignment $(X,X_0,L)\mapsto (X,X_0,L)^a$ defines a functor from the category of admissible triples to that of analytic pseudo-adic spaces. 
We often regard $L$ as a closed subscheme with reduced scheme structure. 
By taking $L$ to be $X_0$, we recover the adic space associated to the admissible pair $(X,X_0)$, that is, we have $(X,X_0,X_0)^a=(X,X_0)^a$. 

Let $(X,X_0,L)$ be an admissible triples and $U$ denote the open subscheme $X\setminus X_0$. 
Similarly to the case of admissible pairs, the associated analytic pseudo-adic space $X^a$ comes with the two natural morphisms of sites;
\begin{itemize}
\item the specialization morphism $\lambda_X=\lambda_{(\wh X,L)}\colon X^a_\et\to L_\et$, 
\item $\varphi_X\colon X^a_\et\to U_\et$ defined to be the composite $X^a_\et=(X,X_0,L)^a_\et\to (X,X_0)^a_\et\to U_\et$.
\end{itemize}

Important examples of admissible triples are those coming from schemes of finite type over a microbial valuation ring. 
Recall that a valuation ring $B$ is called microbial if $B$ has a prime ideal of height one (\cite[Definition 1.1.4]{Hub96}). 
For instance, any finitely dimensional valuation ring is microbial. 
\begin{defn}[{cf.\ \cite[Example 1.9.2.i]{Hub96}}]
\label{val triple}
Let $B$ be a microbial valuation ring. 
The admissible triple $(S,S_0,L_S)$ associated to $B$ is defined as follows. 
Put $S=\Spec B$, let $S_0$ be the closed set defined by vanishing of the unique height one prime ideal $\mathfrak p$ of $B$, and $L_S$ the singleton consisting of the closed point of $S$. 
Note that, as a finitely generated ideal defining $S_0$, we can take the ideal generated by a nonzero element of $\mathfrak p$. 
\end{defn}

\begin{rem}\label{punc}
\bn
\item Let $a$ be a nonzero element of $\mathfrak p$. 
Note that the $a$-adic topology on $B$ coincides with the valuation topology, and hence the formal completion of $S$ along the ideal generated by $a$ is isomorphic to the formal spectrum $\Spf B$ with $B$ being equipped with the valuation topology. 

\item Let $K$ denote the field of fraction of $B$ and $S^a$ the pseudo-adic space associated to $(S,S_0,L_S)$. 
Then the analytification morphism $\varphi_S\colon S^a\to \Spec K$ gives an equivalence of topoi \cite[Proposition 2.3.10]{Hub96}. 
In particular, if $K$ is separably closed, then the topos $S^a$ is equivalent to the punctual topos.
\en
\end{rem}

\begin{exa}\label{geom pt}
Let $(Y,Y_0,L_Y)$ be an admissible triple and $Y^a$ denote the associated pseudo-adic space. 
Let $\xi=(\Spa(\kappa(\xi),\kappa(\xi)^+),\{*\})\to Y^a$ be a geometric point, i.e, $\kappa(\xi)^+$ is a microbial valuation ring with separably closed field of fractions $\kappa(\xi)$ and $*$ is the closed point of $\Spa(\kappa(\xi),\kappa(\xi)^+)$  (\cite[Definition 2.5.1]{Hub96}). 
Let $(S,S_0,L)$ denote the pseudo-adic space associated to the microbial valuation ring $\kappa(\xi)^+$. 
Then $\xi$ is identified with the pseudo-adic space $S^a$  associated to the admissible triple $(S,S_0,L)$.
Note that we have a natural morphism $(S,S_0,L_S)\to (Y,Y_0,L_S)$ of admissible triples. 
The associated morphism $S^a\to  Y^a$ is canonically identified with the geometric point $\xi\to Y^a$. 
\end{exa}

\subsection{Adic spaces and vanishing topoi}
\label{til lambda}
For an admissible triple $(X,X_0,L)$, we construct a canonical morphism $\til\lambda_X\colon X^a\to L\vani_XX$ of topoi (cf.\ \cite[6.4]{Ill06}). 
We denote the complement of $X_0$ by $U$. 
Then we have a natural diagram of topoi
\ben\label{2}\bxy{
X^{a}\ar[r]^{\varphi_X}\ar[d]_{\lambda_X}&U\ar[d]^{j}\\
L\ar[r]^i&X
}\exy\een
with a natural $2$-map $\tau\colon j\varphi_X\to i\lambda_X$ defined as follows (cf.\ the proof of \cite[Theorem 3.5.13]{Hub96}): 
For an object $V$ of $X_\et$, we have a natural morphism
\be
\lambda_X^{-1}i^{-1}V=V^a\to V\times_{X}X^a=\varphi_X^{-1}j^{-1}V
\ee
from \cite[1.9.5]{Hub96}, which induces, for a sheaf $\mc F$ on $(X^a)_\et$, a morphism 
\be
(j_*\varphi_{X*}\mc F)(V)=\mc F(\varphi_X^{-1}j^{-1}V)\to \mc F(\lambda_X^{-1}i^{-1}V)=(i_*\lambda_{X*}\mc F)(V).
\ee
Thus, by the universal property of the vanishing topos $L\vani_XU$, the above diagram induces a canonical morphism of topoi
\be
\widetilde\lambda_{X}\colon X^{a}\to L\vani_XU.
\ee
We denote the composite $X^a\to L\vani_XU\to L\vani_XX$ also by $\til\lambda_X$ when there is no risk of confusion.

\section{Compactly supported direct image for rigid analytic varieties}
\label{sec rigid}
Our main interest in this section is the compactly supported direct image $Rf^a_!$ for a morphism $f\colon (X,X_0,L_X)\to (Y,Y_0,L_Y)$ of admissible triples (Definition \ref{triple}) with $f\colon X\to Y$ compactifiable, where $f^a$ denotes the induced morphism $X^a\to Y^a$ of analytic pseudo-adic spaces. 

Recall that a compactification of $f\colon X\to Y$ induces a compactification of $f^a\colon X^a\to Y^a$. 
More precisely, let $\bar f\colon \bar X\to Y$ be a proper morphism with an open immersion $j\colon X\to \bar X$ such that $f=\bar fj$. 
Let $\bar X_0$ be the pullback $\bar f^{-1}(Y_0)$ and $L_{\bar X}$ be the closure of $L_X$ in $\bar X_0$, so that we have morphisms $(X,X_0,L_X)\overset{j}{\to}(\bar X,\bar X_0,L_{\bar X})\overset{\bar f}{\to}(Y,Y_0,L_Y)$ of admissible triples whose composite is $f$. 
Let $\bar X^a$ denote the pseudo-adic space associated to $(\bar X,\bar X_0,L_{\bar X})$. 
Then $j^a\colon X^a\to\bar X^a$ is an open immersion and $\bar f^a\colon \bar X^a\to Y^a$ is a proper morphism (this can be checked using \cite[Corollary 1.3.9, Lemma 1.3.10]{Hub96}, see also \cite[Lemma 3.5]{Mie06}). 
Thus, by the definition of the compactly supported direct image \cite[Definition 5.4.4]{Hub96}, we have an identification $Rf^a_!=R\bar f^a_*\circ j^a_!$ of functors $D^+(X^a,\Lambda)\to D^+(Y^a,\Lambda)$, for a torsion ring $\Lambda$.  

\subsection{Cohomology of fibers via nearby cycle}
Let $(X,X_0,L)$ be an admissible triple (Definition \ref{triple}) and $U$ the complement of $X_0$. 
For a torsion ring $\Lambda$ and an object $\mc F$ of $D^+(U,\Lambda)$, 
we put $\mc F^a=\varphi_X^*\mc F$. 
The diagram (\ref{2}) induces a canonical morphism
\ben\label{ptwise}
i^*Rj_*\mc F\to R\lambda_{X*}\mc F^a,
\een
which is an isomorphism by the proof of \cite[Corollary 3.5.15]{Hub96} (cf. \cite[Theorem 6.5.4]{Fuj95}).

Here, we give an interpretation of the canonical morphism (\ref{ptwise}) via vanishing topos. 
We consider the diagram of topoi
\ben\label{diagram}\bxy{
X^a\ar[r]^-{\til\lambda_{X}}\ar[rd]_{\lambda_X}
&L\vani_XU\ar[r]^-{p_2}\ar[d]^{p_1}
&U\ar[d]^j
\\
&L\ar[r]^{i}
&X
}\exy\een
with canonical $2$-morphisms $jp_2\to ip_1$ and $p_1\til\lambda_{X}\cong\lambda_X$. 
This diagram induces the morphisms
\ben\label{seq of morphisms}
i^*Rj_*\mc F\to Rp_{1*}p_2^*\mc F\to R\lambda_{X*}\mc F^a,
\een
whose composite is identified with the morphism (\ref{ptwise}). 

\begin{lem}\label{interpretation}
The two morphisms in (\ref{seq of morphisms}) are isomorphisms. 
\end{lem}

%
\begin{proof}
The first one is an isomorphism by \cite[Th\'eor\`eme 2.4]{Ill14}, 
and so is the second, since the composite is an isomorphism as mentioned above. 
\end{proof}

\begin{const}\label{over S}
Let $(S,S_0,L_S)$ be the admissible triple associated to a microbial valuation ring with separably closed filed of 	fractions (Definition \ref{val triple}). 
Let $f\colon (X,X_0,L_X)\to (S,S_0,L_S)$ be a morphism of admissible triples with $f\colon X\to S$ being locally of finite type. 
Then we have the associated morphism $f^a\colon X^a\to S^a$ of pseudo-adic spaces, which is locally of finite type. 

We denote the generic point (resp.\ closed point) of $S$ by $\eta$ (resp.\ $s$). 
For a torsion ring $\Lambda$ and a sheaf $\mc F$ of $\Lambda$-modules on $X_{\eta,\et}$, we consider the nearby cycle complex $R\psi_S\mc F=i^*Rj_*\mc F$ over $S$, where $j$ is the natural open immersion $X_\eta\to X$ and $i$ is the natural closed immersion $X_s\to X$.  
We denote by $\mc F^a$ the pullback of $\mc F$ by the analytification morphism $\varphi_X\colon X^a_\et\to X_{\eta,\et}$. 
Then the canonical isomorphism $R\psi_S\mc F\cong R\lambda_{X*}\mc F^a$ induces a canonical isomorphism 
\be
R\Gamma(L_X,R\psi_S\mc F)\cong R\Gamma(X^a,\mc F^a).\ee
\end{const}

\begin{const}\label{over S with c}
In the situation of Construction \ref{over S}, we assume that the morphism $f\colon X\to S$ is compactifiable. 
We recall the construction of a canonical isomorphism from \cite[Theorem 5.7.8]{Hub96}
\ben\label{comparison}
R\Gamma_c(L_X,R\psi_S\mc F)\cong R\Gamma_c(X^a,\mc F^a).
\een
We take a proper morphism $\bar f\colon \bar X\to S$ with an open immersion $j\colon X\to \bar X$ such that $f=\bar fj$. 
Let $\bar X_0$ be the pullback $\bar f^{-1}(S_0)$ and $L_{\bar X}$ be the closure of $L_X$ in $\bar X_0$, so that we have morphisms $(X,X_0,L_X)\overset{j}{\to}(\bar X,\bar X_0,L_{\bar X})\overset{\bar f}{\to}(S,S_0,L_S)$ of admissible triples whose composite is $f$. 
The induced morphism $\bar f^a\colon \bar X^a\to S^a$ is proper, and $j^a\colon X^a\to \bar X^a$ is an open immersion, as recalled in the beginning of this section. 
We consider the diagram of topoi
\bx{
X^a\ar[r]^{\lambda_X}\ar[d]^{j^a}
&L_X\ar[d]^j
\\\bar X^a\ar[r]^{\lambda_{\bar X}}
&L_{\bar X}.
}\ex
The isomorphism $R\psi_S\mc F\cong R\lambda_{X*}\mc F^a$ and the canonical isomorphism $j_!R\lambda_{X*}\cong R\lambda_{\bar X*}j^a_!$ from \cite[Corollary 3.5.11]{Hub96} induce 
\be
R\Gamma_c(L_X,R\psi_S\mc F)&\cong&R\Gamma_c(L_X,R\lambda_{X*}\mc F^a)
=R\Gamma(L_{\bar X},j_!R\lambda_{X*}\mc F^a)
\\&\cong&R\Gamma(L_{\bar X},R\lambda_{\bar X*}j^a_!\mc F^a)
\\&\cong&R\Gamma(\bar X^a,j^a_!\mc F^a)=R\Gamma_c(X^a,\mc F^a).
\ee
\end{const}

\subsection{Compactly supported direct image via nearby cycles over general bases}

We globalize Huber's construction of the isomorphism between the \'etale cohomology of an algebraizable rigid analytic variety and the nearby cycle cohomology. 

We frequently use the fact that the \'etale topos of any pseudo-adic space has enough points (\cite[Proposition 2.5.5]{Hub96}, see also \cite[Proposition 2.5.17]{Hub96}). 

\begin{const}
Let $f\colon (X,X_0,L_X)\to (Y,Y_0,L_Y)$ be a morphism of admissible triples (Definition \ref{triple}) and let $f^a\colon X^a\to Y^a$ denote the induced morphism of analytic pseudo-adic spaces. 
We consider the natural diagram of topoi
\ben\label{rig and vani}\bxy{
X^a\ar[dd]_{f^a}\ar[r]^-{\til\lambda_X}
&L_X\vani_XX\ar[d]^{\id\vani f}\ar[r]^-{p_2}
&X
\\&L_X\vani_YY\ar[d]^{f\vani\id}
\\Y^a\ar[r]_-{\til\lambda_Y}
&L_Y\vani_YY.
}\exy\een
Here $\til\lambda_X$ and $\til\lambda_Y$ denote the morphisms defined in \S\ref{til lambda}. 
We denote the natural morphism $L_X\vani_XX\to L_Y\vani_YY$ by $\ola f$. 
The diagram above induces the base change morphism $\til\lambda_Y^*R\ola f_*\to Rf^{a}_*\lambda_X^*$. 

Let $\Lambda$ be a ring and $\mc F$ be an object of $D^+(X,\Lambda)$. 
Note that we have canonical isomorphisms 
$R\Psi_f\mc F\cong 
R(\id\vani f)_*R\Psi_\id\mc F\cong R(\id\vani f)_*p_2^*\mc F$ from Remark \ref{p2comp}. 
Thus, the base change morphism $\til\lambda_Y^*R\ola f_*\to Rf^{a}_*\til\lambda_X^*$ induces 
\be
\til\lambda_Y^*R(f\vani\id)_*R\Psi_f\mc F\cong \til\lambda_Y^*R\ola f_*p_2^*\mc F\to Rf^{a}_*\til\lambda_X^*p_2^*\mc F\cong Rf^{a}_*\mc F^a.
\ee
\end{const}

We also give a variant for higher direct image with compact support.  
For the construction, we use the following lemma.
\begin{lem}\label{j of rig and vani}
Let $j\colon (X,X_0,L_X)\to (\bar X,\bar X_0,L_{\bar X})$ be a morphism of admissible triples with $j\colon X\to \bar X$ being an open immersion. 
Let $j^a\colon X^a\to \bar X^a$ denote the induced open immersion of analytic pseudo-adic spaces. 
We consider the diagram of topoi
\bx{
X^a\ar[r]^{j^a}\ar[d]^{\til\lambda_X}
&\bar X^a\ar[d]^{\til\lambda_{\bar X}}
\\
L_X\vani_XX\ar[r]^{\ola j}
&L_{\bar X}\vani_{\bar X}\bar X.
}\ex
Let $\Lambda$ be a ring. 
Then the natural morphism
\be
j^a_!\til\lambda_X^*\to \til\lambda_{\bar X}^*\ola j_!
\ee
of functors $D^+(L_X\vani_XX,\Lambda)\to D^+(\bar X^a,\Lambda)$ is an isomorphism. 
\end{lem}
\begin{proof}
Let $\mc F$ be an object of $D^+(L_X\vani_XX,\Lambda)$. 
We show that, for each geometric point $\xi$ of $\bar X^a$, the morphism $(j^a_!\til\lambda_X^*\mc F)_\xi\to (\til\lambda_{\bar X}^*\ola j_!\mc F)_\xi$ induced on the stalks is an isomorphism. 
For this, it suffices to show that $\xi$ lies over $X^a$ if $x\gets\xi$ lies over $L_X\vani_XX$, where $x\gets \xi$ denotes, by abuse of notation, the image of $\xi$ by $\til\lambda_{\bar X}$. 
Since $x\gets\xi$ lies over $L_X\vani_XX$ if and only if $x$ lies over $L_X$, the assertion follows from the fact that $\lambda_{\bar X}^{-1}(L_X)=X^a$ as a subset of $\bar X^a$. 
\end{proof}

\begin{const}\label{rel comparison}
Let $f\colon (X,X_0,L_X)\to (Y,Y_0,L_Y)$ be a morphism of admissible triples (Definition \ref{triple}) and let $f^a\colon X^a\to Y^a$ denote the induced morphism. 
We now assume that $f\colon X\to Y$ is compactifiable and choose a proper morphism $\bar f\colon \bar X\to Y$ with an open immersion $j\colon X\to \bar X$ such that $f=\bar fj$. 
Let $\bar X_0$ be the pullback $\bar f^{-1}(Y_0)$ and $L_{\bar X}$ be the closure of $L_X$ in $\bar X_0$, so that we have morphisms $(X,X_0,L_X)\overset{j}{\to}(\bar X,\bar X_0,L_{\bar X})\overset{\bar f}{\to}(Y,Y_0,L_Y)$ of admissible triples whose composite is $f$. 

Let $\bar X^a$ denote the pseudo-adic space associated to $(\bar X,\bar X_0,L_{\bar X})$. 
Then $j^a\colon X^a\to\bar X^a$ is an open immersion and $\bar f^a\colon \bar X^a\to Y^a$ is a proper morphism, 
as recalled in the beginning of this section. 

Let $\Lambda$ be a torsion ring. 
Then we have the following sequence of canonical morphisms of functors $D^+(L_X\vani_XX,\Lambda)\to D^+(Y^a,\Lambda)$;
\be
\til\lambda_Y^*R\ola f_!
\cong
\til\lambda_Y^*R\ola{\bar f}_*\ola j_!
\to
R\bar f^a_*\til\lambda_{\bar X}^*\ola j_!
\cong
R\bar f^a_*j^{a}_!\til\lambda_X^*=Rf^a_!\til\lambda_X^*,
\ee
where 
the first isomorphism comes from Lemma \ref{j and push}, the second is the base change morphism induced by the diagram (\ref{rig and vani}) for $\bar f\colon \bar X\to Y$, and the third one comes from Lemma \ref{j of rig and vani}. 

Let $\mc F$ be an object of $D^+(X,\Lambda)$. 
By considering $p_2^*\mc F$, we obtain a canonical morphism 
\be
\theta_f\colon \til\lambda_Y^*R(f\vani\id)_!R\Psi_f\mc F\cong \til\lambda_Y^*R\ola f_!p_2^*\mc F\to Rf^a_!\mc F^a.
\ee 


\end{const}



To state the compatibility of the canonical morphism $\theta_f$ with base change, we make the following construction. 
\begin{const}

Let 
\bx{
(X',X'_0,L_{X'})\ar[r]^{g'}\ar[d]^{f'}
&(X,X_0,L_X)\ar[d]^{f}
\\(Y',Y'_0,L_{Y'})\ar[r]^g
&(Y,Y_0,L_Y)
}\ex
be a commutative diagram of admissible triples. 
We assume that $f\colon X\to Y$ is compactifiable and that the diagram is Cartesian, i.e, the diagram \bx{
X'\ar[r]^{g'}\ar[d]^{f'}
&X\ar[d]^{f}
\\Y'\ar[r]^g
&Y,
}\ex
is a Cartesian diagram of schemes and ${f'}^{-1}(L_{Y'})\cap{g'}^{-1}(L_{X})=L_{X'}$. 

Then, for a torsion ring $\Lambda$ and an object $\mc F$ of $D^+(X,\Lambda)$, we have the base change morphism 
\be
\alpha\colon \ola g^*R(f\vani\id)_!R\Psi_f\mc F\to R(f'\vani\id)_!R\Psi_{f'}({g'}^*\mc F)
\ee
defined in Construction \ref{bc xi}. 
On the other hand, we have a Cartesian diagram 
\bx{
{X'}^a\ar[r]^{{g'}^a}\ar[d]^{{f'}^a}
&X^a\ar[d]^{f^a}
\\{Y'}^a\ar[r]^{g^a}
&Y^a
}\ex
of pseudo-adic spaces, which induces the base change morphism
\be
\beta\colon (g^a)^*Rf^a_!\mc F^a\to R{f'}^a_!{{g'}^a}^*\mc F^a.
\ee
Thus, we have the diagram
\ben\label{bc of comp}\bxy{
\til\lambda_{Y'}^*{\ola g}^*R(f\vani\id)_!R\Psi_f\mc F\ar[r]^{\til\lambda_{Y'}^*\alpha}\ar[d]^{(g^a)^*\theta_{f}}
& \til\lambda_{Y'}^*R(f'\vani\id)_!R\Psi_{f'}({g'}^*\mc F)\ar[d]^{\theta_{f'}}
\\
(g^a)^*Rf^a_!\mc F^a\ar[r]^\cong_\beta
&R{f'}^a_!({g'}^a)^*\mc F^a
}\exy\een
in $D^+({Y'}^a,\Lambda)$. 
\end{const}

\begin{lem}\label{comm w bc}
\begin{enumerate}
\item The diagram (\ref{bc of comp}) is commutative. 
\item If the formation of $R\Psi_f\mc F$ commutes with base change, then the morphism $\alpha$ is an isomorphism.
\item The morphism $\beta$ is an isomorphism if one of the following conditions is satisfied:
\begin{enumerate}
\item $\Lambda$ is killed by an integer invertible in $Y$,
\item $(Y',Y'_0,L_{Y'})$ is the admissible triple associated to a geometric point $\xi$ of $Y^a$ (Example \ref{geom pt}) and $g\colon(Y',Y'_0,L_{Y'})\to(Y,Y_0,L_Y)$ is the natural morphism. 
\end{enumerate}
\end{enumerate}
\end{lem}
\begin{proof}
1. 
By Construction \ref{bc bc}, the morphism $\alpha$ is identified with the base change morphism 
\be
\ola{g}^*R\ola f_!p_2^*\mc F\to R\ola{f'}_!\ola{g'}^*p_2^*\mc F
\ee
induced by the diagram
\bx{
L_{X'}\vani_{X'}X'\ar[r]^{\ola g'}\ar[d]^{\ola f'}
&L_X\vani_XX\ar[d]^{\ola f}
\\L_{Y'}\vani_{Y'}Y'\ar[r]^{\ola g}
&L_Y\vani_YY.
}\ex
Thus, both of $\theta_{f'}\circ\til\lambda_{Y'}^*\alpha$ and $\beta\circ({g^a})^*\theta_{f}$ are identified with the base change morphism induced by the diagram
\bx{
{X'}^a\ar[r]\ar[d]
&L_X\vani_XX\ar[d]
\\{Y'}^a\ar[r]
&L_Y\vani_YY.
}\ex
\\
\\
2. This is clear from the definition of $\alpha$.
\\
3. This follows from \cite[Theorem 5.3.9, Corollary 5.3.10]{Hub96}. 
\end{proof}

\begin{lem}\label{case of pt}
Let $(S,S_0,L_S)$ be the admissible triple associated to a microbial valuation ring with separably closed filed of 	fractions (Definition \ref{val triple}). 
Let $f\colon (X,X_0,L_X)\to (S,S_0,L_S)$ be a morphism of admissible triples with $f\colon X\to S$ being compactifiable. 
Let $\Lambda$ be a torsion ring and $\mc F$ an object of $D^+(X,\Lambda)$. 
Then the morphism 
\ben\label{rig nearby c}
\til\lambda_S^*R(f\vani\id)_!R\Psi_f\mc F\to Rf^{a}_!\mc F^a
\een
in $D^+(S^a,\Lambda)$ is identified with the morphism (\ref{comparison}) in Construction \ref{over S with c} via the equivalence $R\Gamma(S^a,-)\colon D^+(S^a,\Lambda)\cong D^+(\Lambda)$ (Remark \ref{punc}.2), and hence is an isomorphism.  
\end{lem}
\begin{proof}
Since $S^a$ and $L_S\vani_S\eta$, where $\eta$ is the generic point of $S$, are equivalent to the punctual topos, the morphism (\ref{rig nearby c}) is identified with the morphism
\ben\label{Gamma}
R\Gamma_c(L_X\vani_XX_\eta,p_2^*\mc F)\to R\Gamma_c(X^a,\mc F^a)
\een
induced by the adjunction morphism $\theta\colon p_2^*\mc F\to R\til\lambda_{X*}\til\lambda_X^*p_2^*\mc F=R\til\lambda_{X*}\mc F^a$. 
We consider the diagram of topoi
\bx{
X^a\ar[r]^-{\til\lambda_{X}}\ar[rd]_{\lambda_X}
&L_X\vani_XX_\eta\ar[r]^-{p_2}\ar[d]^{p_1}
&X_\eta\ar[d]^j
\\&L_X\ar[r]^i
&X.
}\ex
The morphism (\ref{Gamma}) is identified with the morphism
\be
R\Gamma_c(L_X,Rp_{1*}p_2^*\mc F)\to R\Gamma_c(L_X,R\lambda_{X*}\mc F^a)\cong R\Gamma_c(X^a,\mc F^a) 
\ee
induced from the above diagram. 
Then the assertion follows from Lemma \ref{interpretation}. 
\end{proof}

\begin{thm}\label{globalization}
Let $f\colon (X,X_0,L_X)\to (Y,Y_0,L_Y)$ be a morphism of admissible triples (Definition \ref{triple}) with $f\colon X\to Y$ being compactifiable. 
Let $\Lambda$ be a torsion ring and $\mc F$ an object of $D^+(X,\Lambda)$. 
We assume that the formation of $R\Psi_f\mc F$ commutes with base change. 
Then the morphism
\be
\theta_f\colon \til\lambda_Y^*R(f\vani\id)_!R\Psi_f\mc F\to Rf^{a}_!\mc F^a
\ee
from Construction \ref{rel comparison} is an isomorphism. 
\end{thm}

\begin{proof}
We prove that $\theta_f$ induces an isomorphism of the stalks at each geometric point $\xi$ of $Y^a$. 
Example \ref{geom pt} shows that we have the natural morphism $(S,S_0,L_S)\to (Y,Y_0,L_Y)$ of admissible triples such that the associated morphism $S^a\to Y^a$ of pseudo-adic spaces is canonically identified with the geometric point $\xi\to Y^a$. 
By forming a Cartesian diagram 
\bx{
(X',X'_0,L_{X'})\ar[r]\ar[d]
&(X,X_0,L_X)\ar[d]
\\(S,S_0,L_S)\ar[r]
&(Y,Y_0,L_Y)
}\ex
of admissible triples and applying Lemma \ref{comm w bc} to it, the problem is reduced to the case where $(Y,Y_0,L_Y)=(S,S_0,L_S)$. 
Then the assertion follows from Lemma \ref{case of pt}.
\end{proof}

\begin{proof}[Proof of Theorem \ref{construction}]
Follows from Proposition \ref{globalization} and \cite[Th\'eor\`eme 2.1]{Org06}. 
\end{proof}

\begin{rem}
It seems impossible to give a similar comparison isomorphism for $Rf^a_*$ (see Remark \ref{not qc}). 
\end{rem}

\section{Quasi-constructibility}
\label{sec blfin}
In \S\ref{ssec blfin}, we state and deduce from Theorem \ref{globalization} our finiteness result (Theorem \ref{modfin}). 
In \S\ref{ssec d2}, we focus on the case where the target is the $2$-dimensional disk to give a concrete consequence Proposition \ref{lc} of Theorem \ref{modfin}. 

\subsection{Modification and finiteness}
\label{ssec blfin}
\begin{defn}
Let $\sf X$ be an analytic pseudo-adic space which is quasi-compact and quasi-separated. 
Let $\Lambda$ be a noetherian ring. 
\bn\item Let $\msf F$ be a sheaf of $\Lambda$-modules on $\sf X$. 
We say that $\msf F$ is {\it strictly quasi-constructible} if there exist a finite partition $\msf X=\coprod_{i\in I}\msf L_i$ by locally closed constructible subsets $\msf L_i$ and a finite partiion $\msf X=\coprod_{j\in J} \msf Z_j$ by Zariski locally closed subsets $\msf Z_j$ such that the restriction of $\msf F$ to $\msf L_i\cap \msf Z_j$ is locally constant of finite type for every $(i,j)\in I\times J$. 
\item Let $\msf F$ be an object of $D^b(\msf X,\Lambda)$. We say that $\msf F$ is {\it strictly quasi-constructible }if the all cohomology sheaves $\mc H^i(\msf F)$ are strictly quasi-constructible. 
\en
\end{defn}

\begin{rem}
If $\msf F$ is strictly quasi-constructible, then $\msf F$ is quasi-constructible in the sense of \cite[Definition 1.1]{Hub98}. 
\end{rem}

\begin{lem}\label{pullback by lambda}
Let $(X,X_0,L_X)$ be an admissible triple (Definition \ref{triple}) with $X$ being quasi-compact and quasi-separated and let $X^a$ denote the associated analytic pseudo-adic space (\S3.1). 
We consider the natural morphism $\til\lambda_X\colon X^a\to L_X\vani_XU$ defined in \S\ref{til lambda}, where $U$ is the complement of $X_0$. 
Let $\Lambda$ be a noetherian ring and $\mc K$ a constructible sheaf of $\Lambda$-modules on $L_X\vani_XU$. 
Then $\til\lambda_X^*\mc K$ is strictly quasi-constructible. 
\end{lem}

\begin{proof}
We may assume that $L_X=X_0$. We take finite partitions $X_0=\coprod_{i\in I}L_i$ and $U=\coprod_{j=1}^rZ_j$ by locally closed subsets $L_i\subset X_0$ and $Z_j\subset U$ such that $\mc K$ is locally constant on $L_i\vani_XZ_j$. 
Let $\msf L_i$ be the pullback of $L_i$ by the map $\lambda_X\colon \msf X\to X_0$ and $\msf Z_i$ the pullback of $Z_i$ by the map $\varphi_X\colon \msf X\to U$. 
Then each $\msf L_i$ is a constructible locally closed subset, and each $\msf Z_j$ is a Zariski locally closed subset. 
Further, the restriction of $\til\lambda^*\mc K$ to $\msf L_i\cap\msf Z_j$ is locally constant. 
Thus, $\til\lambda_X^*\mc K$ is quasi-constructible. 
\end{proof}


\begin{defn}
A morphism $X'\to X$ of schemes is called a {\it modification }if it is proper and surjective, if there exists a dense open subscheme $U$ of $X$ such that the restriction $X'\times_XU\to U$ is an isomorphism, and if every generic point (i.e, maximal point) of $X'$ is sent to a generic point of $X$.  
\end{defn}

\begin{thm}\label{modfin}
Let $f\colon (X,X_0,L_X)\to (Y,Y_0,L_Y)$ be a morphism of admissible triples (Definition \ref{triple}). 
We assume that $Y$ is a coherent scheme having only finitely many irreducible components and that $f\colon X\to Y$ is separated and of finite presentation. 
Let $n$ be an integer invertible on $Y$ and $\mc F$ be an object of $D^b_c(X,\Z/n\Z)$. 
Then there exists a modification $\pi\colon Y'\to Y$ such that the complex ${\pi^a}^*Rf^a_!\mc F^a$ is strictly quasi-constructible. 
\end{thm}

\begin{proof}
Recall that the formation of $Rf^a_!$ commutes with base change (\cite[Theorem 5.3.9]{Hub96}). 
Hence, we can freely replace $Y$ by a modification. 
By \cite[The\'eor\`eme 2.1]{Org06}, there exists a modification $Y'\to Y$ such that the formation of $R\Psi_{f'}\mc F'$ commutes with base change, where $f'$ is the base change $X'=X\times_YY'\to Y'$ and $\mc F'$ is the pullback of $\mc F$ by the projection $X'\to X$. 
Thus, we may assume that the formation of $R\Psi_f\mc F$ commutes with base change. 
Then, by Proposition \ref{globalization}, we have $\til\lambda_Y^*R(f\vani\id)_!R\Psi_f\mc F\cong Rf^a_!\mc F^a$. 
On the other hand, by Lemma \ref{c of family of psi}, $R(f\vani\id)_!R\Psi_f\mc F$ is constructible. 
Thus, by Lemma \ref{pullback by lambda}, $Rf^a_!\mc F^a$ is strictly quasi-constructible. 
\end{proof}

In particular, Theorem \ref{blfin} follows. 

\begin{rem}\label{not qc}
The above theorem does not hold for the direct image without support. 
In general, the complex $Rf^a_*\mc F^a$ is not of the form $\til\lambda_Y^*\mc G$ for some object $\mc G$ of $D^b_c(L_Y\vani_YY,\Lambda)$. 
As showed in \cite[Example 1.1]{Hub98b}, $Rf^a_*\mc F^a$ is not quasi-constructible in general. 
\end{rem}

\subsection{Quasi-constructible sheaves on the $2$-dimensional disk}
\label{ssec d2}
In this subsection, we fix a non-archimedean field $k$ and use the notations in Example \ref{rigid}. 

\begin{lem}
\label{lc locus}
Let $\pi\colon S'\to S=\Spec k^\circ[s,t]$ be a modification and $\msf F$ a sheaf of $\Lambda$-modules on $S^\rig\cong\Spa(k\la s,t\ra,k^\circ\la s,t\ra)$, for a noetherian ring $\Lambda$.  
We assume that ${\pi^\rig}^{*}\msf F$ is strictly quasi-constructible. 
Then there exist an integer $\nu\ge0$ and $\delta\in |k^\times|$ such that $\msf F$ is locally constant on the open subset $\{x\in S^\rig\mid 0<|s(x)|\le|t^\nu(x)| \text{ and }|t(x)|\le\delta\}$. 
\end{lem}

\begin{lem}\label{bl an}
Let $\pi\colon S'\to S=\Spec k^\circ[s,t]$ be a modification. 
Then there exist an integer $\nu\ge0$, an open neighborhood $\msf U\subset \Spa(k\la\frac{s}{t^\nu},t\ra,k^\circ\la\frac{s}{t^\nu},t\ra)$ of the origin, and a commutative diagram 
\bx{\msf U\ar[r]\ar[rd]_{\pi_\nu}&(S')^\rig\ar[d]^{\pi}\\&S^\rig,}\ex
where $\pi_\nu$ is the morphism induces by the natural map $k\la s,t\ra\to k\la\frac{s}{t^\nu},t\ra$. 
\end{lem}

\begin{proof}
Applying Lemma \ref{alg bl} below to the regular local ring $k[s,t]_{(s,t)}$, we can find a commutative diagram 
\bx{
U=\Spec k[\frac{s}{t^\nu},t,\frac{1}{f}]\ar[r]\ar[rd]&S'\ar[d]\\
&S}\ex
for some $\nu\ge1$ and $f\in k[\frac{s}{t^\nu},t]$ which does not belong to the maximal ideal $(\frac{s}{t^\nu},t)$. 
We take $\msf U$ to be the fiber product $U\times_SS^\rig$ (\cite[Proposition 3.8]{Hub94}). 
Since we have a natural isomorphism $(S')^\rig\cong S'\times_SS^\rig$ (\cite[Proposition 1.9.6]{Hub96}), we obtain a desired commutative diagram by applying $\times_SS^\rig$ to the above commutative diagram. 
\end{proof}

\begin{lem}\label{alg bl}
Let $A$ be a regular local ring of dimension $2$ with regular parameters $s,t$. 
Let $\pi\colon X\to \Spec A$ be a modification. 
Then there exist an integer $\nu\ge0$ and a commutative diagram 
\bx{\Spec A[\frac{s}{t^\nu}]_{(\frac{s}{t^\nu},t)}\ar[r]\ar[rd]_{\pi_\nu}&X\ar[d]^{\pi}\\&\Spec A,}\ex
where $\pi_\nu$ is the natural morphism. 
\end{lem}

\begin{proof}
This follows from the valuative criterion of proper morphisms and the fact that 
$\varinjlim_\nu A[\frac{s}{t^\nu}]_{(\frac{s}{t^\nu},t)}$ is a valuation ring. 
\end{proof}

\begin{lem}\label{qc subset}
Let $\msf U\subset \Spa(k\la s,t\ra,k^\circ\la s,t\ra)$ be an open neighborhood of the origin and $\msf Z\subset U$ be a Zariski closed subset. 
Then there exist an integer $\nu\ge0$ and $\delta\in|k^\times|$ such that 
$\msf U\setminus \msf Z$ contains the open subset 
\be
\{x\in \Spa(k\la s,t\ra,k^\circ\la s,t\ra)\mid 0<|s(x)|\le|t^\nu(x)|\text{ and }|t(x)|\le\delta\}. 
\ee
\end{lem}
\begin{proof}
We may assume that $\msf U=\Spa(k\la s,t\ra,k^\circ\la s,t\ra)$. 
We take a nonzero function $f\in k\la s,t\ra$ such that the zero of $f$ contains $\msf Z$, i.e, a nonzero element in an ideal defining $\msf Z$.  
We may assume that $\msf Z$ is the Zariski closed subset defined by $f$. 
Since we do not care about points $x$ with $|s(x)|=0$, we may assume that $f$ is not divided by $s$. 
Then, by replacing $f$ by $cf$ for some $c\in k^\times$ if needed, we can write 
\be
f=at^\nu+t^{\nu+1}g+sh
\ee 
with $a\in k^\times$, $g\in k^\circ\la t\ra$, and $h\in k^\circ\la s,t\ra$. 
Take $\delta\in|k^\times|$ such that $\delta<|a|$. 
Let $x\in\Spa(k\la s,t\ra,k^\circ\la s,t\ra)$ be a point which satisfies $|s(x)|\le |t^{\nu+1}(x)|$ and $|t(x)|\le\delta$. 
Then we have 
\be
|(t^{\nu+1}g+sh)(x)|\le|t^{\nu+1}(x)|\le \delta\cdot|t^\nu(x)|<|at^\nu(x)|,
\ee
and hence, $|f(x)|\ne0$. 
This proves the assertion. 
\end{proof}

\begin{proof}[Proof of Lemma \ref{lc locus}]
By Lemma \ref{bl an}, the problem is reduced to the case where $\mc F$ is strictly quasi-constructible. 
Then there exists a constructible subset $\msf L$ which contains the origin of $\B^2_k$ and a Zariski closed subset $\msf Z\subset \B^2_k$ such that $\mc F$ is locally constant on $\msf L\setminus(\msf L\cap \msf Z)$. 
Note that, since the origin is a maximal point, $\msf L$ contains an open neighbor hood of the origin. 
Then the assertion follows from Lemma \ref{qc subset}.
\end{proof}

\begin{cor}
\label{lc}
Let $f\colon X\to S=\Spec k^\circ[s,t]$ be a morphism separated of finite presentation and $\mc F$ be a constructible sheaf of $\Z/n\Z$-modules on $X$, for an integer $n$ invertible in $k^\circ$.  
Then there exist an integer $\nu\ge0$ and $\delta\in|k^\times|$ such that $Rf^\rig_!\mc F^\rig$ is locally constant on the open subset $S^\rig=\Spa(k\la s,t\ra,k^\circ\la s,t\ra)$ defined by $0<|s|\le|t^\nu|$ and $|t|\le\delta$. 
\end{cor}

\section{Tubular neighborhoods}\label{sec tube}
In this section, we fix a non-archimedean field $k$ and assume that $k$ is of characteristic zero. 
We also fix an integer $n$ invertible in $k^\circ$. 
We introduce some notations. 
The rigid affine line $\Spec k[t]\times_{\Spec k}\Spa(k,k^\circ)$ (\cite[Proposition 3.8]{Hub94}) is denoted by $\A$. 
For $r\in |k^\times|$, the disk (resp.\ punctured disk) of radius $r$, i.e, the open subset $\{x\in \A\mid |t(x)|\le r\}$ (resp.\ $\{x\in\A\mid 0<|t(x)|\le r\}$) is denoted by $\B(r)$ (resp.\ $\B^*(r)$). 
We put $\B=\B(1)$ and $\B^*=\B^*(1)$. 

Let $p\colon P\to \Spec k^\circ[t]$ be a morphism of schemes separated of finite type and $X$ a closed subscheme of $P$ defined by a global function $g\in\mc O_P(P)$. 
We denote the natural morphism $X\to \Spec k[t]$ by $f$. 
For an integer $\nu\ge0$ and an element $\e\in |k^\times|$, we consider the closed subset $X^\rig_{\nu,\e}=\{x\in P^\rig\times_{\B}\B^*\mid |g(x)|<\e\cdot|t^\nu(x)|\}$ and denote the natural morphism $X^\rig_{\nu,\e}\to \B^*$ of pseudo-adic spaces by $f^\rig_{\nu,\e}$.

\begin{prop}\label{tube}
Let $\mc F$ be a constructible sheaf of $\Z/n\Z$-modules on $P$. 
There exist an integer $\nu\ge0$ and an element $\e_0\in|k^\times|$ satisfying the following property: 
For every $\e\in|k^\times|$ with $\e\le\e_0$, there exists an element $\delta\in |k^\times|$ such that the natural morphism 
\be
Rf^\rig_{\nu,\e!}\mc F^\rig\to Rf^\rig_!\mc F^\rig
\ee
is an isomorphism on $\B^*(\delta)$. 
\end{prop}


We recall a structure theorem of finite \'etale coverings of annuli due to L\"utkebohmert. 
For this, we introduce a terminology:
\begin{defn}
Let $r,R\in |k^*|$ be elements with $r\le R$ and $X\to \B(r,R)=\{x\in \A\mid r\le|t(x)|\le R\}$ be a finite \'etale covering. 
We say $X\to \B(r,R)$ is of {\it Kummer type }if it is isomorphic to the cover 
\be
\B(r^{1/d},R^{1/d})\to\B(r,R);z\to z^d.
\ee 
\end{defn}

Then L\"utkebohmert's structure theorem states the following. 
\begin{thm}[{\cite[Theorem 2.2]{Lut93}}]
\label{Lut}
Let $d\ge1$ be an integer and $k$ be an algebraically closed complete non-archimedean field extension of $\Q_p$. 
Then there exists an element $\beta\in|k^\times|$ such that every finite \'etale covering $X\to \B(r,R)$, for every $r\le R\in |k^\times|$ with $\beta^{-1}r\le\beta R$, is of Kummer type over $\B(\beta^{-1}r,\beta R)$.
\end{thm}

For $r\in |k^\times|$, the closed disk (resp.\ punctured closed disk) of radius $r$, i.e, the subset $\{x\in \A\mid |t(x)|< r\}$ (resp.\ $\{x\in\A\mid 0<|t(x)|< r\}$) is denoted by $\D(r)$ (resp.\ $\D^*(r)$). 
We put $\D=\D(1)$ and $\D^*=\D^*(1)$. 
We regard $\D(r)$ and $\D^*(r)$ as pseudo-adic spaces. 
\begin{lem}[{c.f. (II) in the proof of \cite[Theorem 2.5]{Hub98}}]
\label{Kum}
Assume that $k$ is algebraically closed. 
There exists an element $\e_0\in |k^\times|$ satisfying the following condition: 
Let $\mc M$ be a locally constant constructible sheaf of $\Z/n\Z$-modules on $\B^*$. 
Then for every $\e\le\e_0$, we have $R\Gamma_c(\D^*(\e),\mc M)=0$. 
\end{lem}
We follow the proof of \cite[Lemma 6.12]{Ito20}. 
\begin{proof}
Take $\e_0$ to be an element $\beta$ as in Theorem \ref{Lut}. 
Since $H^i_c(\D^*(\beta),\mc M)\cong\varinjlim_{r\in|k^\times|}H^i_c(\D(\beta)\setminus\D(r),\mc M)$, it suffices to prove that $R\Gamma_c(\D(\beta)\setminus\D(r),\mc M)=0$. 
By Theorem \ref{Lut}, we may assume that $\mc M$ is a constant sheaf. 
Then the explicit description $R\Gamma_c(\D(r),\Z/n\Z)\cong \Z/n\Z$ as in \cite[Example 0.2.5]{Hub96} shows the assertion. 
\end{proof}

\begin{lem}\label{gen lc}
Let $\B^2_k=\Spa(k\la s,t\ra,k^\circ\la s,t\ra)$ and $\nu\ge0$ an integer. 
For an element $\e\in|k^\times|$, we consider the subsets $\msf U_\e=\{x\in \B^2_k\mid 0<|s(x)|\le\e\cdot|t^\nu(x)|\}$ and $\msf U_\e'=\{x\in\B^2_k\mid0<|s(x)|<\e\cdot|t^\nu(x)|\}$, which we regard as pseudo-adic spaces.  
We denote the second projection $\msf U'_\e\to\B=\Spa(k\la t\ra,k^\circ\la t\ra)$ by $\pr'_{\e,2}$. 
Let $\mc M$ be a locally constant constructible sheaf of $\Z/n\Z$-modules on $\msf U_1$. 
Then there exists an element $\e_0\in|k^\times|$ satisfying the following condition: 
For every element $\e\in |k^\times|$ with $\e\le\e_0$, there exists an element $\delta\in|k^\times|$ such that the restriction of $R\pr'_{\e,2!}(\mc M|_{\msf U'_\e})$ to $\B^*(\delta)$ vanishes.
\end{lem}
\begin{proof}
We may assume that $k$ is algebraically closed. 
Considering the isomorphism $\varphi_\nu\colon \B\times\B^*\cong\B\times\B^*;(s,t)\mapsto(s/t^\nu,t)$ of adic spaces, 
we may assume that $\nu=0$. 
We denote the natural immersion $\msf U'_\e\to \B^2_k$ by $u$. 
Then the sheaf $u_!\mc M$ is strictly quasi-constructible (in particular, quasi-constructible). 
Thus, by \cite[Theorem 2.1]{Hub98} and \cite[1.2.iv)]{Hub98}, there exists an element $\delta\in |k^\times|$ such that the restriction of $R\pr'_{\e,2!}\mc M$ to $\B^*(\delta)$ is locally constant of finite type. 
Note that, for a classical point $x$ of $\B^*(\delta)$, we have $(R\pr'_{\e,2!}\mc M)_x\cong R\Gamma_c(\D(\e),\mc M|_{\D(\e)\times x})$ by \cite[Theorem 5.3.9]{Hub96}. 
Take $\e_0$ to be an element as in Lemma \ref{Kum}. 
Then, if $\e\le\e_0$, then the stalk $(R\pr'_{\e,2!}\mc M)_x\cong R\Gamma_c(\D(\e),\mc M|_{\D(\e)\times x})$, and hence $R\pr'_{\e,2!}\mc M$ vanishes. 
\end{proof}

\begin{proof}[Proof of Proposition \ref{tube}]
We consider the morphism $g\colon P\to S=\Spec[s,t]$ over $\Spec k[t]$ defined by $s\mapsto g\in \mc O_{P}(P)$. 
By Corollary \ref{lc}, we can apply Lemma \ref{gen lc} to the restriction of $R^ig^\rig_!\mc F^\rig$ to the subset $\{x\in S^\rig\mid 0<|s(x)|\le|t^\nu(x)|\}$, which proves the assertion.  
\end{proof}

\begin{rem}
It would be straightforward to extend Proposition \ref{tube} to the case where the closed subscheme $X$ is not necessarily defined by one global equation, i.e, the case of general closed subschemes $X$. 

It might be also possible to remove the assumption on the characteristic of the base field $k$ by an argument similar to the proof of \cite[Lemma 7.1]{Ito20}. 
\end{rem}


\end{document}